\documentclass[a4paper,12pt]{article}

\usepackage{amsfonts}
\usepackage{amscd,color}
\usepackage{amsmath,amsfonts,amssymb,amscd,algorithmicx,setspace}
\usepackage{indentfirst,graphicx,epsfig}
\usepackage{graphicx,psfrag}
\input{epsf}
\usepackage{graphicx}
\usepackage{epstopdf}
\usepackage{caption}
\usepackage{indentfirst}

\newtheorem{thm}{Theorem}[section]
\newtheorem{df}[thm]{Definition}

\newtheorem{lem}[thm]{Lemma}

\newtheorem{prob}{Problem}
\newenvironment {proof} {\noindent{\em Proof.}}{\hspace*{\fill}$\Box$\par\vspace{4mm}}

\title{\textbf{ Conflict-free connectivity: algorithm and complexity\footnote{Supported by NSFC No.11871034, 11531011 and NSFQH No.2017-ZJ-790.}}}
\author{{\small Meng Ji, Xueliang Li, Xiaoyu Zhu } \\
{\small  Center for Combinatorics and LPMC}\\
{\small Nankai University, Tianjin 300071, P.R. China}\\
{\small Email: jimengecho@163.com; lxl@nankai.edu.cn; zhuxy@mail.nankai.edu.cn}\\
}
\date{}
\begin{document}
\maketitle
\begin{abstract}

A path in an(a) edge(vertex)-colored graph is called \emph{a conflict-free path} if there exists a color used on only one of its edges(vertices). An(A) edge(vertex)-colored graph is called \emph{conflict-free (vertex-)connected} if there is a conflict-free path between each pair of distinct vertices. We call the graph $G$ \emph{strongly conflict-free connected }if there exists a conflict-free path of length $d_G(u,v)$ for every two vertices $u,v\in V(G)$. And the \emph{strong conflict-free connection number} of a connected graph $G$, denoted by $scfc(G)$, is defined as the smallest number of colors that are required to make $G$ strongly conflict-free connected. In this paper, we first investigate the question: Given a connected graph $G$ and a coloring $c: E(or\ V)\rightarrow \{1,2,\cdots,k\} \ (k\geq 1)$ of the graph, determine whether or not $G$ is, respectively, conflict-free connected, vertex-conflict-free connected, strongly conflict-free connected under coloring $c$. We solve this question by providing polynomial-time algorithms. We then show that it is NP-complete to decide whether there is a k-edge-coloring $(k\geq 2)$ of $G$ such that all pairs $(u,v)\in P \ (P\subset V\times V)$ are strongly conflict-free connected. Finally, we prove that the problem of deciding whether $scfc(G)\leq k$ $(k\geq 2)$ for a given graph $G$ is NP-complete.   \\[2mm]
\end{abstract}

\section{Introduction}

All graphs mentioned in this paper are simple, undirected and finite. We follow book \cite{BM} for
undefined notation and terminology. Coloring problems are important parts of graph theory. In recent years, there have appeared a number of colorings raising great
concern due to their wide applications in real world. We list a few well-known colorings here. The first of such would be the rainbow connection coloring, which is stated as follows. A path in an edge-colored graph is called a {\it rainbow path} if all its edges have distinct colors.
An edge-colored connected graph is called {\it rainbow connected} if there is a rainbow path between every pair of distinct vertices
in this graph. For a connected graph $G$, the smallest number of colors needed to make $G$ rainbow connected is called the
{\it rainbow connection number} of $G$, denoted by $rc(G)$. This concept was first introduced by Chartrand et al. in \cite{CJMZ}. Chakraborty et al. have proved in \cite{CFMY} that deciding whether $rc(G)=2$ is NP-complete.

Inspired by the rainbow connection coloring, the concept of proper connection coloring was independently posed by Andrews et al. in \cite{ALLZ} and Borozan et al. in \cite{BFGMMMT}, its only difference from rainbow connection coloring is that distinct colors are only required for adjacent edges instead of all edges on the path. For an edge-colored connected graph $G$, the smallest number of colors required to give $G$ a proper connection coloring is called the \emph{proper connection number} of $G$, denoted by $pc(G)$. Almost in the same time, Caro and Yuster \cite{CY} introduced the concept of \emph{monochromatic connection coloring}. A path in an edge-colored graph $G$ is a \emph{monochromatic path} if all the edges of the path are colored the same. The graph $G$ is called \emph{monochromatically connected} if any two distinct vertices of G are connected by a monochromatic path. The \emph{monochromatic connection number} of $G$, denoted by $mc(G)$, is the {\bf maximum} number of colors such that $G$ is monochromatically connected. A lot of results have been obtained since these concepts were introduced.

In this paper, we focus on the conflict-free (vertex-)coloring. The hypergraph version of this concept was first introduced
by Even et al. in \cite{ELRS}. A hypergraph $H$ is a pair $H=(X,E)$ where $X$ is the set of vertices, and $E$ is the set of non-empty
subsets of $X$ called edges. The coloring was motivated to solve the problem of assigning frequencies to different base stations
in cellular networks. There are a number of base stations and clients in the network. Each base station is a vertex in the hypergraph
which needs to be allocated to a frequency. Different frequencies stand for different colors in a vertex-colored hypergraph.
Every client is moveable, so it can be in the range of lots of base stations. Thus each client is a set of many vertices, i.e,
clients represent edges. For each client, in order to make connection with one of the base station in the range, there must be
at least one base station with a unique frequency in the range for fear of interference. Unnecessarily many different frequencies
can be expensive, so this situation may be converted to a conflict-free vertex-coloring problem of a hypergraph seeking for
the minimum number of colors which is defined as the {\it conflict-free chromatic number} of the hypergraph.

Later on, Czap et al. \cite{CJV} introduced the concept of \emph{conflict-free connection} of graphs on the basis of the earlier hypergraph version. A path in an edge-colored graph $G$ is called a \emph{conflict-free path} if there is a color appearing only once on the path. The graph $G$ is called \emph{conflict-free connected} if there is a conflict-free path between each pair of distinct vertices of $G$. The minimum number of colors required to make $G$ conflict-free connected is called the \emph{conflict-free connection number} of $G$, denoted by $cfc(G)$.

As a natural counterpart of the conflict-free connection, Li et al. \cite{LZZMZJ} introduced the concept of \emph{conflict-free vertex-connection} of graphs. A path in a vertex-colored graph is called a \emph{conflict-free path} if it has at least one vertex with a unique color on the path. A vertex-colored graph is called \emph{conflict-free vertex-connected} if there is a conflict-free path between every pair of distinct vertices of $G$. The minimum number of colors required to make $G$ conflict-free vertex-connected is called the \emph{conflict-free vertex-connection number} of $G$, denoted by $vcfc(G)$.

There have been many results on the conflict-free (vertex-)connection coloring due to its theoretical and practical significance.
\begin{thm}{\upshape\cite{CJV, DLLMZ}\label{noncomplete}}
Let $G$ be a noncomplete 2-edge-connected graph. Then $cfc(G)=2$.
\end{thm}
\begin{lem}{\upshape \cite{CJV}}\label{path}
If $P_n$ is a path on $n$ vertices, then $cfc(P_n)=\lceil
\log_2 n\rceil $.
\end{lem}
\begin{thm}{\upshape \cite{CJV}}\label{diameter}
If $T$ is a tree on $n$ vertices with maximum degree $\Delta(T)\geq 3$ and diameter $diam(T)$, then
$$
\max\{\Delta(T),\log_2 diam(T)\}\leq cfc(T)\leq \frac{(\Delta(T)-2)\log_2 n}{\log_2\Delta(T)-1}.
$$
\end{thm}

\begin{thm}{\upshape\cite{LZZMZJ}}
Let $G$ be a connected graph of order at least $3$. Then $vcfc(G)=2$ if and only
if $G$ is $2$-connected or $G$ has only one cut-vertex.
\end{thm}

\begin{thm}{\upshape\cite{LZZMZJ}}
Let $P_n$ be a path on $n$ vertices. Then $vcfc(P_n)=\lceil\log_2(n+1)\rceil$.
\end{thm}

\begin{thm}{\upshape\cite{LZZMZJ}}
Let $T$ be a tree of order $n\geq 3$ and diameter $d(T)$. Then
\begin{eqnarray*}
\max\{\chi(T),\lceil\log_{2}(d(T)+2)\rceil\}\leq vcfc(T)\leq
\log_{\frac{3}{2}}n.
\end{eqnarray*}
\end{thm}

In \cite{LZZMZJ}, Li et al. posed, as a conjecture, and Li and Wu in \cite{LW} verified the following result.

\begin{thm}{\upshape\cite{LW}}
For any connected graph $G$ with $n$ vertices, $vcfc(G)\leq vcfc(P_n)$.
\end{thm}

Hence, they got a tight upper bound for the conflict-free vertex-connection number of connected graphs of order $n$.
In the same paper, Li and Wu posed, as a conjecture, and Chang et al. in \cite{CJLZ} verified the following result.

\begin{thm}{\upshape\cite{CJLZ}}
For a tree $T$ of order $n$, $cfc(T)\geq cfc(P_n)=\lceil\log_2n\rceil$.
\end{thm}

As can be seen in Theorem 1.1, the conflict-free connection number of graphs without cut-edges has been obtained. Thus determining the value of $cfc(G)$ for graphs $G$ with cut-edges becomes the main task. Trees are extremal such graphs for which every edge is a cut-edge. For a tree $T$ we can build a hypergraph $H$ as follows. The hypergraph $H_{EP}(T)=(\mathcal{V},\mathcal{E})$ has $\mathcal{V}(H_{EP})=E(T)$ and $\mathcal{E}(H_{EP})=\{E(P)|$
$P$ is a path of $T\}$. One can easily see that the conflict-free chromatic number of the hypergraph $H$ is just the conflict-free connection number of $T$.
For more results we refer to \cite{CDHJLS, CHLMZ, CJLZ, DLLMZ}. Nevertheless, most of them are about the graph structural characterizations. The graph structural analytic method may be more useful to handle graphs with certain characterizations such as 2-edge-connected graph and some given graph classes. But a polynomial algorithm is applicable to all general graphs. However very few results on this have been obtained for now. Thus we address the computational
aspects of the (strong)conflict-free (vertex-)connection colorings in this paper.

First of all, we pose the definition of the strong conflict-free connection as follows:
\begin{df}
A path in an edge-colored graph $G$ is called a\emph{ conflict-free path }if there has a color appearing only once on the path. The graph $G$ is called \emph{strongly conflict-free connected} if there is a conflict-free path which is the shortest among all $u$-$v$-paths between each pair of distinct vertices $(u,v)$ of $G$ and the corresponding edge-coloring is called the \emph{strong conflict-free coloring} of $G$. The minimum number of colors required to make $G$ strongly conflict-free connected is called the \emph{strong conflict-free connection number} of $G$, denoted by $scfc(G)$.
\end{df}

Combining all kinds of colorings
we present above, it is natural to ask such a question:

\begin{prob}
Given an integer $k\geq 1$ and a connected graph $G$, is it $NP$-hard or polynomial-time solvable to answer each of
the following questions ?

\noindent(a) Is $rc(G)\leq k$ ?\\
(b) Is $pc(G)\leq k$ ?\\
(c) Is $mc(G)\geq k$ ?\\
(d) Is $cfc(G)\leq k$? (Is $vcfc(G)\leq k$ ? for the vertex version)\\
(e) Is $scfc(G)\leq k$? (can be also referred to as the {\bf k-strong conflict-free connectivity problem} in the following context)
\end{prob}

For general graphs, Ananth et al. proved in \cite {AN} that Question $(a)$ is $NP$-hard.
Chakraborty et al. proved in \cite{CFMY} that Question $(a)$ is $NP$-complete even if $k=2$.
The answers for Questions $(b), \ (c)$, $(d)$ and $(e)$ remain unknown. For a tree $T$, Question $(a)$ is easy
since $rc(T)=n-1$, and Question $(b)$ is also easy since $pc(T)=\Delta (T)$, where $n$ is the order of $T$
and $\Delta(T)$ is the maximum degree of $T$. However, the complexity for Question $(d)$ is unknown
even if $G$ is a tree $T$.

Actually, {\bf Problem 1} is equivalent to the following statement:

\begin{prob}
Given an integer $k\geq 1$ and a connected graph $G$, determine whether there is a $k$-edge (or vertex) coloring to make $G$

\noindent(a) rainbow connected.\\
(b) proper connected.\\
(c) monochromatically connected.\\
(d) conflict-free connected (or conflict-free vertex-connected).\\
(e) strongly conflict-free connected.
\end{prob}

The following is a weaker version for {\bf Problem 1}:

\begin{prob}
Given a connected graph $G$ with $n$ vertices and $m$ edges and a coloring $c: E \ (or\ V)\rightarrow \{1,2,\cdots,k\} \ (k\geq 1)$ of the graph,
for each pair of distinct vertices $u,v$ of $G$, determine whether there is a path $P$ between $u,v$ such that

\noindent(a) $P$ is a rainbow path.\\
(b) $P$ is a proper path.\\
(c) $P$ is a monochromatic path.\\
(d) $P$ is a conflict-free path.\\
(e) $P$ is a strong conflict-free path.
\end{prob}

For general graphs, Chakraborty et al. proved in \cite{CFMY} that  Question $(a)$ is $NP$-complete. Recently, Ozeki \cite{Ozeki}
confirmed that  Question $(b)$ is polynomial-time solvable. It is not difficult to see that Question $(c)$ can also be solved in
polynomial-time, just by checking all subgraphs each being induced by the set of edges with a same color.

There is also another weaker version for {\bf Problem 1}(e).

\begin{prob}[k-subset strong conflict-free connectivity problem]
Given a graph $G$ and a set $P\subset V\times V$, decide whether there is an edge-coloring of $G$ with $k$ colors such that all pairs $(u,v)\in P$ are strongly conflict-free connected.
\end{prob}

The article is arranged as follows: Next section, we will provide two
polynomial-time algorithms for {\bf Problem 3} $(d)$ and {\bf Problem 3} $(e)$. In section 3, we present the complexity results of the strong conflict-free connection version as proving that it is NP-complete to answer {\bf Problem 4} when $k\geq 2$ and {\bf Problem 1}(e) when $k\geq2$.

\section{ Polynomial-time algorithms}

Before presenting our main theorem for Question $(d)$ in {\bf Problem 3}, some auxiliary lemmas are needed.

\begin{lem}{\upshape\cite{CJV}}\label{path}
Let $u,v$ be distinct vertices and $e=xy$ be an edge of a $2$-connected graph $G$. Then there is a $u\text{-}v$-path in $G$ containing the edge $e$.
\end{lem}

Let $x$ be a vertex and $Y$ be a vertex-set of a connected graph $G$, then a family of $k$ internally disjoint
$(x,Y)$-paths whose terminal vertices are pairwise distinct is referred to as a $k\text{-}fan$ from $x$ to $Y$.

With this, we have the famous Fan Lemma.

\begin{lem}
Let $G$ be a $k$-connected graph, and $x$ be a vertex of $G$, and let $Y\subseteq V\backslash \{x\}$ be a set of at least $k$
vertices of $G$. Then there exists a $k\text{-}fan$ in $G$ from $x$ to $Y$.
\end{lem}

For a connected graph $G$, a vertex of $G$ is called {\it a separating vertex} if its removal will leave $G$
splitting into two nonempty connected subgraphs. We call the graph {\it nonseparable} if it is connected
without separating vertices. A {\it block} of the graph is a subgraph which is nonseparable and maximal
in this property. We can construct a bipartite graph $B(G)$ for every connected graph $G$ as follows:
let $V(B(G))=(\mathcal{B},S)$ where $\mathcal{B}$ represents the set of all blocks in $G$ and $S$ is
the set of separating vertices. The block $B\in \mathcal{B}$ and vertex $s\in S$ are adjacent if and only
if $s\in B$ in $G$. It is clear that $B(G)$ is also a tree, we call it the {\it block tree}.

\begin{lem}\label{edge}
For a connected graph $G$, let $u,v\in V(G)$, $st\in E(G)$. Then there is no $u\text{-}v$-path
containing edge $st$ if and only if there exists a vertex $z$ such that neither $u$ nor $v$ is
connected to $s$ or $t$ in the graph $G-z$.
\end{lem}

\noindent{\bf Proof of sufficiency:} Suppose there exists a $u\text{-}v$-path containing $st$.
Then obviously $z$ must appear at least twice in this path, a contradiction.

\noindent{\bf Proof of necessity:} We claim that $G$ is not $2$-connected since otherwise
Lemma \ref{path} will lead to a contradiction.

Assume that $st\in B_1$, $u\in B_2$ and $v\in B_3$ where $B_i \ (i=1,2,3)$ is the block of
$G$. Then $B_1=B_2=B_3$ cannot happen since otherwise a $u\text{-}v$-path containing $st$ can
be found according to Lemma \ref{path}, a contradiction. If $B_2=B_3$, then the removal of
any separating vertex on the path of $B(G)$ between $B_1$ and $B_2$ will leave neither $u$ nor $v$
connected to $s$ or $t$. Consider the case that $B_2\neq B_3$. We claim that $B_1$ is not on the path
between $B_2$ and $B_3$ in $B(G)$, since otherwise a $u\text{-}v$-path can be chosen to go
through $st$ by applying Lemma \ref{path} to $B_1$, also a contradiction. At last, we consider
the deletion of the first separating vertex on the path of $B(G)$ from $B_1$ to $B_2$, this will
cause the disconnections we want. $\Box$

With a similar proof, one can get the corresponding lemma for vertex version.

\begin{lem}\label{vertex}
For a connected graph $G$, let $u,v,w\in V(G)$. Then there is no $u\text{-}v$-path
containing vertex $w$ if and only if there exists a vertex $z\neq w$ such that
neither $u$ nor $v$ is connected to $w$ in the graph $G-z$.
\end{lem}

The famous Depth-First Search(DFS) will be used in our algorithm. For a graph
$G$ with $n$ vertices and $m$ edges, the DFS starts from a root vertex $x$ and
goes as far as it can along a path, after that, it backtracks until finding
a new path and then explores it. The algorithm stops when all vertices of $G$
have been explored. As is well known, the time complexity for DFS is $\mathcal{O}(n+m)$.

\begin{thm}
There exists a polynomial-time algorithm to determine Question $(d)$ in {\bf Problem 3}.
The complexity for the edge version is at most $\mathcal{O}(n^3m^2)$,
and the complexity for the vertex version is at most $\mathcal{O}(n^4m)$.
\end{thm}

\noindent{\bf Proof of the edge version:} Given $k\geq1$ and a connected graph $G$ with an
edge-coloring $c:E(G)\rightarrow\{1,2,\cdots,k\}$, let $E_i(i=1,2,\cdots,k)$ be the edge-set
containing all edges colored with $i$. We present our algorithm below:

\begin{algorithmic}
\State \rule[-0.1\baselineskip]{\textwidth}{1pt}
\textbf{Algorithm 1: Determining whether an edge-colored graph is conflict-free connected}
\begin{spacing}{0.6}
\State \rule[0.6\baselineskip]{\textwidth}{1pt}
\end{spacing}
\noindent\textbf{Input:} A given integer $k\geq1$ and a connected graph $G$ with $n$ vertices, $m$ edges and
an edge-coloring $c: E(G)\rightarrow \{1,2,\cdots,k\}$.\\
\textbf{Output:} Whether $G$ is conflict-free connected or not.\\
\textbf{1:} Check if there is an unselected pair of distinct vertices in $G$. If so,
pick one pair $(u,v)$, go to \textbf{2}; otherwise, go to \textbf{8}.\\
\textbf{2:} Set $i=0$, go to \textbf{3}.\\
\textbf{3:} Check if $i\leq k-1$. If so, $i:=i+1$, $G':=G-E_i$, go to \textbf{4}; otherwise, go to \textbf{9}.\\
\textbf{4:} For $(u,v)$, determine if there is an unselected edge $e$ in $E_i$. If so, pick $e=st$, set $G'':=G'+e$, go to \textbf{5}; otherwise, go to \textbf{3}.\\
\textbf{5:} Check if $u,v$ and $st$ are connected in $G''$. If so, go to \textbf{6}; otherwise, go to \textbf{4}.\\
\textbf{6:} For $(u,v)$ and $st$, determine if there is an unselected vertex in $G''$. If so, pick one vertex $z$, go to \textbf{7}; otherwise, go to \textbf{1}.\\
\textbf{7:} Determine if neither $u$ nor $v$ is connected to $s$ or $t$ in $G''-z$. If so, go to \textbf{4}; otherwise, go to \textbf{6}.\\
\textbf{8:} Return: $G$ is conflict-free connected under coloring $c$.\\
\textbf{9:} Return: $G$ is not conflict-free connected under coloring $c$.\\
\rule[0.7\baselineskip]{\textwidth}{1pt}
\end{algorithmic}

Let us first prove the algorithm above is correct. If for a pair of distinct vertices $(u,v)$,
there is no conflict-free path between them, then for any edge $e$ in $G$, there is
no $u\text{-}v$-path in $G-E_{c(e)}+e$ containing $e$. Thus according to Lemma \ref{edge},
for each $e$, there must be a vertex $z$ (step \textbf{6}) such that neither $u$ nor $v$ is
connected to $s$ or $t$ in $G''-z=G-E_{c(e)}+e-z$. As a result, after traversing every
edge in $G$, it comes to step \textbf{4}, then step \textbf{3} and finally step \textbf{9}
obtaining the right result that $G$ is not conflict-free connected.

If for $(u,v)$, there is a conflict-path between them, then there must exist an edge $e$
such that for any vertex $z$ in $G$, either $u$ or $v$ is connected to $s$ or $t$ in
$G''-z=G-E_{c(e)}+e-z$. Therefore, after repeating steps \textbf{7} and \textbf{6}
for some times, the running process will come to step \textbf{1} and then examine
the next pair of vertices. If all pair of vertices have been examined, it will announce
that $G$ is conflict-free connected. This shows the correctness of our algorithm.

For a fixed pair of vertices $(u,v)$ and a fixed edge $e=st$, to examine step \textbf{5},
we only need to apply the DFS algorithm appointing $s$ as the root vertex. Then for any
vertex $z$ of $G$, again apply the DFS algorithm to step \textbf{7}. Consequently we
get that the complexity is $\mathcal{O}((n+m)n+n+m)=\mathcal{O}(nm)$. Since there
are $\mathcal{O}(n^2)$ pair of vertices and $m$ edges in $G$, the overall
complexity is at most $\mathcal{O}(n^3m^2)$.

\noindent{\bf Proof of the vertex version:} With Fan Lemma and Lemma \ref{vertex},
it actually has analogous analysis with the edge version. The differences are as follows:
$(i)$ $V_i \ (1\leq i\leq k)$ shall take the place of $E_i \ (1\leq i\leq k)$, and $(ii)$ we will pick a vertex this time instead of an edge in step \textbf{4}
.
Because of this, an $m$ will be replaced by an $n$ in the complexity for
the edge version, so the time complexity for the vertex version is $\mathcal{O}(n^4m)$.

Besides, for a picked pair of vertices $(u,v)$, if $c(u)=c(v)$, then
the vertex-set $V_{c(u)}$ is not needed to consider in step \textbf{3}
since $c(u)$ can never be the unique color on any $u\text{-}v$-path; if
$c(u)\neq c(v)$, any vertex of $(V_{c(u)}\backslash u) \ (\text{or}\ (V_{c(v)}\backslash v))$
is not needed to add back after removing $(V_{c(u)}\backslash u) \ (\text{or}\ (V_{c(v)}\backslash v))$
from $G$ (like in step \textbf{4}) because the unique color has already exists on $u \ (\text{or}\ v)$.
This saves some operations compared to the algorithm for the edge version. Thus the complexity
for the vertex version is at most $\mathcal{O}(n^4m)$.  $\Box$

For Question $(e)$ in {\bf Problem 3}, we also get a polynomial-time algorithm in which the Breadth-First Search(BFS) is used. For a graph $G$ with $n$ vertices and $m$ edges, the BFS starts from a root vertex $x$ and explores all the neighbors of the vertices at the present level before moving to the next depth level. The algorithm stops when all vertices of $G$
have been explored. As is well known, the time complexity for BFS is $\mathcal{O}(n+m)$.

Before presenting our algorithm, we think it is necessary to give a definition.

\begin{df}
For a vertex $u$ in a connected graph $G$, it is obvious that any edge $e=st$ must have $|d_G(u,s)-d_G(u,t)|\leq1$. So $e$ is called a \emph{vertical edge} of $u$ if $|d_G(u,s)-d_G(u,t)|=1$ and a \emph{horizontal edge} of $u$ otherwise.
\end{df}

\begin{thm}\label{strong polynomial-time algorithm}
There exists a polynomial-time algorithm to determine Question $(e)$ in {\bf Problem 3}.
The complexity  is at most $\mathcal{O}(n^2m^2)$.
\end{thm}

\begin{proof} Given $k\geq1$ and a connected graph $G$ with an
edge-coloring $c:E(G)\rightarrow\{1,2,\cdots,k\}$, let $E_i(i=1,2,\cdots,k)$ be the edge-set
containing all edges colored with $i$. We present our algorithm below:

\begin{algorithmic}
\State \rule[-0.1\baselineskip]{\textwidth}{1pt}
\textbf{Algorithm 2: Determining whether an edge-colored graph is strongly conflict-free connected}
\begin{spacing}{0.6}
\State \rule[0.6\baselineskip]{\textwidth}{1pt}
\end{spacing}
\noindent\textbf{Input:} A given integer $k\geq1$ and a connected graph $G$ with $n$ vertices, $m$ edges and
an edge-coloring $c: E(G)\rightarrow \{1,2,\cdots,k\}$.\\
\textbf{Output:} Whether $G$ is strongly conflict-free connected or not.\\
\textbf{1:} Check if there is an unselected pair of distinct vertices in $G$. If so,
pick one pair $(u,v)$, go to \textbf{2}; otherwise, go to \textbf{6}.\\
\textbf{2:} Set $i=0$, go to \textbf{3}.\\
\textbf{3:} Check if $i\leq k-1$. If so, $i:=i+1$, $G':=G-E_i$, go to \textbf{4}; otherwise, go to \textbf{7}.\\
\textbf{4:} For $(u,v)$, determine if there is an unselected vertical edge $e=st$ with $d_G(u,s)<d_G(u,t)\leq d_G(u,v)$ in $E_i$. If so, set $G'':=G'+e$, go to \textbf{5}; otherwise, go to \textbf{3}.\\
\textbf{5:} Check if $d_G(u,s)=d_{G''}(u,s)$ and $d_{G''}(v,t)=d_G(u,v)-d_G(u,t)$. If so, go to \textbf{1}; otherwise, go to \textbf{4}.\\
\textbf{6:} Return: $G$ is strongly conflict-free connected under coloring $c$.\\
\textbf{7:} Return: $G$ is not strongly conflict-free connected under coloring $c$.\\
\rule[0.7\baselineskip]{\textwidth}{1pt}
\end{algorithmic}

We will prove that the algorithm above is correct. If for a pair of distinct vertices $(u,v)$,
there is no conflict-free shortest path between them, then for any vertical edge $e=st$ with $d_G(u,s)<d_G(u,t)\leq d_G(u,v)$ in $G$, any $u\text{-}v$-path in $G-E_{c(e)}+e$ containing $e$ has length greater than $d_G(u,v)$. Hence there must be $d_G(u,s)\neq d_{G''}(u,s)$ or $d_{G''}(v,t)\neq d_G(u,v)-d_G(u,t)$ in step \textbf{5}. As a result, after traversing every vertical
edge $e=st$ with $d_G(u,s)<d_G(u,t)\leq d_G(u,v)$ in $G$, it comes to step \textbf{4}, then step \textbf{3} and finally step \textbf{7}
obtaining the right result that $G$ is not strongly conflict-free connected.

If for $(u,v)$, there is a conflict-free shortest path between them, then there must exist a vertical edge $e=st$ with $d_G(u,s)<d_G(u,t)\leq d_G(u,v)$ in $G$ such that we can obtain a $u\text{-}v$-path in $G-E_{c(e)}+e$ containing $e$ whose length is equal to $d_G(u,v)$. Then there must be $d_G(u,s)=d_{G''}(u,s)$ and $d_{G''}(v,t)=d_G(u,v)-d_G(u,t)$. Therefore, the running process will come to step \textbf{1} after step \textbf {5} and then examine
the next pair of vertices. If all pair of vertices have been examined, it will announce
that $G$ is strongly conflict-free connected. This shows the correctness of our algorithm.

For a fixed pair of vertices $(u,v)$, firstly we need to apply the BFS algorithm to $G$ designating $u$ as the root to acquire all vertical edge $e=st$ with $d_G(u,s)<d_G(u,t)\leq d_G(u,v)$ in $G$. Then for any fixed edge $e=st$, we only need to apply the BFS algorithm a few more times to $G'$ to examine step \textbf{5}. Consequently we
get that the complexity is $\mathcal{O}(n+m+m(n+m))=\mathcal{O}(m^2)$. Since there
are $\mathcal{O}(n^2)$ pair of vertices in $G$, the overall
complexity is at most $\mathcal{O}(n^2m^2)$.\end{proof}

If one wants to determine whether an edge-colored graph is $k$-subset strongly conflict-free connected, one only need to examine all pair of vertices in $P$ instead of in $V\times V$ in Algorithm 2. Then we immediately have the following theorem:

\begin{thm}\label{polynomial-time for k-subset}
There exists a polynomial-time algorithm to determine whether an edge-colored graph is $k$-subset strongly conflict-free connected.
\end{thm}

\section{Hardness results on strong conflict-free connectivity problems}

When $k=2$, we prove that {\bf Problem 4} is NP-complete in subsection \ref{2-subset}; then when $k\geq3$, by showing that {\bf Problem 4} is still NP-complete, we derive the final result that {\bf Problem 1}(e) is NP-complete in subsection \ref{$k$-strong}.
\subsection{2-subset strong conflict-free connectivity problem}\label{2-subset}
Our main theorem is listed as below:

\begin{thm}\label{thm1}
 For $k=2$, {\bf Problem 4} is NP-complete.
\end{thm}

We first define the following problem.

\begin{prob}[Partial 2-edge-coloring problem]
Given a graph $G=(V,E)$ and a partial 2-edge-coloring $\hat{c}:$ $\hat{E}\rightarrow \{0,1\}$ for $\hat{E}\subset E$, decide whether $\hat{c}$ can be extended to a complete 2-edge-coloring $c:$ $E\rightarrow \{0,1\}$ that makes $G$ strongly conflict-free connected.
\end{prob}

When $k=2$, we first reduce {\bf Problem 5} to {\bf Problem 4}, and then reduce 3-SAT to {\bf Problem 5}, finally Theorem \ref{thm1} is completed since Theorem \ref{polynomial-time for k-subset} implies that {\bf Problem 4} belongs to NP.

\begin{lem}\label{lem2}
For $k=2$, {\bf Problem 5}$\preceq$ {\bf Problem 4}.
\end{lem}

\begin{proof} Given such a partial coloring $\hat{c}$ for $\hat{E}\subset E$, we denote $\hat{E}=\hat{E_1}\cup \hat{E_2}$ where $\hat{E_1}$ contains all edges in $\hat{E}$ colored with 0 and $\hat{E_2}=\hat{E}\setminus \hat{E_1}$. We then extend the original graph $G=(V,E)$ to a graph $G'=(V',E')$, and define a set $P$ of pairs of vertices of $V'$ such that the answer for {\bf Problem 5} with $G$ and $\hat{c}$ as parameters is yes if and only if the answer for {\bf Problem 4} with $G'$ and $P$ as parameters is yes.

Let $[n]$($n=|V|$) be an arbitrary linear ordering of the vertices and $l(v)(v\in V)$ be the number related to $v$ in this ordering. Let $\theta:$ $E\rightarrow V$ be a mapping that maps an edge $e=uv$ to $u$ if $l(u)>l(v)$, and to $v$ otherwise. On the contrary, let $\varepsilon:$ $E\rightarrow V$ be a mapping that maps $e=uv$ to $u$ if $l(u)< l(v)$, and to $v$ otherwise. Let $r=\lceil\frac{n}{2}\rceil$ if $\lceil\frac{n}{2}\rceil$ is odd, otherwise $r=\lceil\frac{n}{2}\rceil+1$. We polynomially construct $G'$ as follows:
its vertex set
\begin{center}
$V'=V\cup V_1\cup V_2\cup V_3$ where\\
$V_1=\{b_1,c,b_2\}$ \\
$V_2=\{c_e:for\ \forall e\in (\hat{E}_1\cup \hat{E}_2)\}$\\
$V_3=\{t^e_1,t^e_2,\cdots,t^e_{r}: for\ \forall e\in (\hat{E}_1\cup\hat{E}_2)\}$
\end{center}
and edge set
\begin{center}
$E'=E\cup E_1\cup E_2\cup E_3$ where\\
$E_1=\{b_1c,b_2c\}$\\
$E_2=\{b_it^e_1, t^e_1t^e_2,\cdots, t^e_{r-1}t^e_r, t^e_rc_e: i\in\{1,2\}, e\in\hat{E}_i\}$\\
$E_3=\{c_e\varepsilon(e):e\in(\hat{E}_1\cup\hat{E}_2)\}$\\
\end{center}

Now we define the set $P$ of pairs of vertices of $V'$:
\begin{center}
$P=\{b_1,b_2\}\cup\{\{u,v\}:u,v\in V, u\neq v\}\cup\{\{c,t^e_1\},\{b_i,t^e_2\},\{t^e_1,t^e_3\},\{t^e_2,t^e_4\},\cdots,\{t^e_{r-2}t^e_r\},\{t^e_{r-1},c_e\},\{t^e_r,\varepsilon(e)\}:i\in \{1,2\},e\in\hat{E}_i\}\cup\{\{c_e,\theta(e)\}:e\in (\hat{E}_1\cup\hat{E}_2)\}$
\end{center}

Now, if there is a strong conflict-free coloring with 2 colors $\pi_c=(E_1,E_2)$ of $G$ which extends $\pi_{\hat{c}}=(\hat{E}_1,\hat{E}_2)$, then we color $G'$ as follows. Every edge $e\in E$ retains coloring $c$: the edge is colored with 0 if it is in $E_1$ and otherwise it is colored with 1. Edges $b_1c, \varepsilon(e)c_e$ for $e\in \hat{E}_2$ are all colored with 0, $b_2c$ and $c_e\varepsilon(e)$ for $e\in \hat{E}_1$ are all colored with 1. Moreover, edges $b_1t^e_1, t^e_1t^e_2, \cdots, t^e_{r-1}t^e_r, t^e_rc_e$ $(e\in \hat{E}_1)$ are assigned the color 1 and 0 alternately and edges $b_2t^e_1, t^e_1t^e_2, \cdots, t^e_{r-1}t^e_r, t^e_rc_e$ $(e\in \hat{E}_2)$ are assigned the color 0 and 1 alternately. One can see that this coloring indeed makes each pair in $P$ strongly conflict-free connected.

On the other direction, we can see that $P$ contains all vertex pairs of $G$ and for each of these pairs, all the shortest paths between it in $G'$ are completely contained in $G$. Thus any 2-edge-coloring of $G'$ that strongly conflict-free connects the pairs in $P$ clearly contains a strong conflict-free coloring of $G$. Also, such a coloring would have to color $cb_1$ and $cb_2$ differently. It would also have to color every $b_1t^e_1(e\in \hat{E}_1)(b_2t^e_1(e\in \hat{E}_2))$ in a color different from that of $cb_1(cb_2)$. By further reasoning, we can see that the colorings of $b_1t^e_1, t^e_1t^e_2, \cdots, t^e_{r-1}t^e_r, t^e_rc_e$ $(e\in \hat{E}_1)$ and $b_2t^e_1, t^e_1t^e_2, \cdots, t^e_{r-1}t^e_r, t^e_rc_e$ $(e\in \hat{E}_2)$ are both alternately. As a result, $c_e\varepsilon(e)(e\in \hat{E}_1)$ must be in a color different from that of $cb_1$ and $c_e\varepsilon(e)(e\in \hat{E}_2)$ is in a color different from that of $cb_2$. Finally, every $e\in \hat{E}_i$ must be assigned the color identical to that of $cb_i$ to make $\theta(e)$ and $c_e$ strongly conflict-free connected. Without loss of generality, we suppose that the edge $cb_1$ is colored with 0. It is clear that this coloring of $G'$ conforms to the original partial coloring $\hat{c}$. This implies that $\hat{c}$ can be extended to a complete 2-edge-coloring $c:$ $E\rightarrow \{0,1\}$ that makes $G$ strongly conflict-free connected.
\end{proof}

\begin{lem}\label{lem3}
3-SAT$\preceq$ {\bf Problem 5}.
\end{lem}

\begin{proof}
Let $\phi:=\bigwedge^l_{i=1}c_i$ be a 3-conjunctive normal form formula over variables \{$x_1,x_2,\cdots,x_n$\}. Then we polynomially construct the graph $G'=(V',E')$ as follows:

\begin{center}
$V'=\{c_i:i\in [l]\}\cup\{x_i:i\in[n]\}\cup\{a\}$\\
$E'=\{x_ic_j:x_i\in c_j\}\cup\{x_ia:i\in[n]\}\cup\{c_ic_j:i,j\in[l]\}\cup\{x_ix_j:i,j\in[n]\}$
\end{center}

Now we give the partial 2-edge-coloring $c'$: edges $\{c_ic_j:i,j\in [l]\}$ and $\{x_ix_j:i,j\in [n]\}$ are assigned the color 0; the edge $x_ic_j\in E'$ is assigned the color 0 if $x_i$ is positive in $c_j$ and color 1 otherwise. Thus only the edges in $\{x_ia:i\in[n]\}$ are left uncolored.

Without loss of generality, we assume that all variables in $\phi$ appear both as positive and as negative, so it only remains to prove that there is an extension $c$ of $c'$ that enables a conflict-free shortest path between $a$ and each $c_i(i\in [l])$ if and only if $\phi$ is satisfiable since there will always be a conflict-free shortest path between any other pair of vertices of $V'$ whatever the extension is. Let $c(x_ia)=x_i(i\in[n])$, one can verify that this relationship does hold. In fact, in a successful extension $c$ of $c'$, the color vector formed by $c(x_ia)(i\in[n])$ can be seen as a solution vector of $\phi$, and vice versa.\end{proof}

\subsection{$k$-strong conflict-free connectivity problem}\label{$k$-strong}

The following is our main theorem.

\begin{thm}\label{thm2}
For $k\geq2$, {\bf Problem 1} $($e$)$ is NP-complete.
\end{thm}

In the following we prove Theorem \ref{thm2} for $k=2$ and for $k\geq3$, separately.

At first let us deal with the case $k=2$. Chakraborty et al. in \cite{CFMY} obtained the following result.

\begin{thm}\upshape\cite{CFMY}\label{cfmy}
Given a graph $G$, deciding if $rc(G) = 2$ is NP-complete. In particular,
computing $rc(G)$ is NP-hard.
\end{thm}

Then we can easily get the following result by the definitions of rainbow connection and conflict-free connection.

\begin{lem}\label{new}
Given a graph $G=(V,E)$, $rc(G)=2$ if and only if $diam(G)=2$ and $scfc(G)=2$.
\end{lem}
\begin{proof}
For a connected graph $G$, if $rc(G)=2$ then obviously $diam(G)=2$. Since $diam(G)=2$ and obviously $scfc(G)\le rc(G)$, we have $2\le scfc(G)\le rc(G)=2$, i.e., $scfc(G)=2$. So, we get that
both $diam(G)=2$ and $scfc(G)=2$.

On the other hand, if $diam(G)=2$, then for each pair of vertices of $G$, the length of every shortest conflict-free path between the two vertices is at most 2,
and so every shortest conflict-free path must be a rainbow path. Since $scfc(G)=2$, then two colors are enough to make $G$ strongly conflict-free connected. So, $rc(G)\le 2$.
Since $diam(G)=2$, then $rc(G)= 2$.
\end{proof}
\begin{thm}\label{problem1}
For $k=2$, {\bf Problem 1} $($e$)$ is NP-complete.
\end{thm}
\begin{proof}
It is NP-complete to decide whether the rainbow connection number of a connected graph is 2 by Theorem \ref{cfmy}. Therefore, deciding whether $scfc(G)=2$ and $diam(G)=2$ is NP-complete by Lemma \ref{new}.
Since it is easy to see that deciding if $diam(G)=2$ can be done in polynomial-time, then deciding if $scfc=2$ must be NP-complete.
\end{proof}

Now we are left to deal with the case $k\ge 3$, Recall the famous NP-complete problem below.

\begin{prob}[k-vertex coloring problem]
Given a graph $G=(V,E)$ and a fixed integer $k$, decide whether there is a k-vertex-coloring for $G$ such that each color class is an independent set.
\end{prob}

Next lemma is necessary for the proof of our theorem.

\begin{lem}\label{conflict-free path}
For $k\geq3$, {\bf Problem 6}$\preceq$ {\bf Problem 4}.
\end{lem}
\begin{proof} Now we polynomially construct a graph $G'=(V',E')$: for a given connected graph $G=(V,E)$, let $V'=V\cup \{x\}$, $E'=\{vx:v\in V\}$, and $P=\{(u,v):uv\in E\}$. It remains to prove that graph $G=(V,E)$ is vertex colorable with $k\geq 3$ colors if and only if graph $G'= (V',E')$ can be k-edge-colored such that there is a conflict-free path of length $d_{G'}(u,v)$ between every pair $(u, v)\in P$.

For one direction, assume that $G$ can be vertex-colored with k colors, we prove that there is an assignment of $k$ colors to the edges of the graph $G'$ that enables a conflict-free path of length $d_{G'}(u,v)$ between every pair $(u, v)\in P$ . We construct a bijection between $V$ and $E'$: $v\in V\rightarrow vx\in E'$. If $i$ is the color assigned to a vertex $v\in V$, then we assign the color $i$ to the edge $xv\in E'$. For any pair $(u,v)\in P$, since $uv\in E$, $xu$ and $xv$ have different colors. Thus, the unique path $u-x-v$ is a conflict-free shortest path between $u$ and $v$. The other direction can be also easily verified according to the bijection above.
\end{proof}

There is exactly one path between every pair of vertices in $P$ since the graph $G'$ constructed in the above proof is a tree. Thus, combining Theorem \ref{polynomial-time for k-subset} with Lemma \ref{conflict-free path}, we get the conclusion immediately:
\begin{thm}\label{k-subset}
For $k\geq 3$, {\bf Problem 4} is NP-complete even when $G$ is a star.
\end{thm}
The following lemma is a consequence of Theorem \ref{thm1} and Theorem \ref{k-subset}:
\begin{thm}
For $k\geq 2$, {\bf Problem 4} is NP-complete.
\end{thm}

\noindent{\bf Proof of Theorem \ref{thm2}:}
For $k=2$, it holds by theorem \ref{problem1}. Then for $k\geq3$, considering Theorem \ref{strong polynomial-time algorithm} and Lemma \ref{conflict-free path}, to prove Theorem \ref{thm2}, we only need to reduce the instances obtained from the proof of Lemma \ref{conflict-free path} to some instances of {\bf Problem 1}(e).
Let $G=(V,E)$ be a star graph with $\hat{V}=\{v_1,v_2,\cdots,v_n\}$ being the leaf vertex set and $a$ being the non-leaf vertex. The vertices of any pair $(v_i,v_j)\in P$ are both the leaf vertices in graph $G$. And we construct a graph $G'$ according to graph $G$ as follows: for every vertex $v_i\in \hat{V}$ we introduce two new vertices $x_{v_i}$ and $x'_{v_i}$, and for every pair of leaf vertices $(u,v)\in (\hat{V}\times \hat{V})\setminus P$ we introduce two new vertices $x_{(u,v)},x'_{(u,v)}$. Then we have:
\begin{flushleft}
    $V'=V\cup V_1\cup V_2$ where\\
    $V_1=\{x_{v_i}:i\in\{1,\cdots,n\}\}\cup\{x_{(v_i,v_j)}:(v_i,v_j)\in (\hat{V}\times \hat{V})\setminus P\}$\\
    $V_2=\{x'_{v_i}:i\in\{1,\cdots,n\}\cup\{x'_{(v_i,v_j)}:(v_i,v_j)\in (\hat{V}\times \hat{V})\setminus P\}$\\
    $E'=E\cup E_1\cup E_2\cup E_3\cup E_4$ where\\
    $E_1=\{v_ix_{v_i}:v_i\in \hat{V},x_{v_i}\in V_1\}$\\
    $E_2=\{v_ix_{(v_i,v_j)},v_jx_{(v_i,v_j)}:(v_i,v_j)\in (\hat{V}\times \hat{V})\setminus P\}$\\
    $E_3=\{xx':x\in V_1,x'\in V_2\}$\\
    $E_4=\{ax':x'\in V_2\}$
\end{flushleft}

Then we need to prove that $G'$ is k-strong conflict-free connected if and only if $G$ is $k$-subset strongly conflict-free connected.

Firstly, there is a two-length path $v_i-x-v_j$ in $G$ for all pairs $(v_i,v_j)\in P$, and this path also occurred in $G'$ which is the unique path of length two in $G'$ between $v_i$ and $v_j$. It implies that if the graph $G'$ is strongly conflict-free colored with $k$ colors, then $G$ has an edge-coloring with $k$ colors such that every pair in $P$ is strongly conflict-free connected.

Secondly, assume that there is a k-edge-coloring $c$ of $G$ such that all pairs in $P$ are strongly conflict-free connected. Then we extend this edge-coloring $c$ of $G$ to an edge-coloring $c'$ of $G'$: $E$ retain coloring $c$; assign color 3 to $uv\in E_1$; assign $v_ix_{(v_i,v_j)}$, $v_jx_{(v_i,v_j)}\in E_2$ the color 1 and 2 respectively. Since subgraph $H=(V_1\cup V_2,E_3)$ is a complete bipartite graph, we choose a perfect matching $M$ of size $|V_1|$, giving the edges in $M$ color 1 and the edges in $E_3\setminus M$ color 2. We then assign the edges $ax'\in E_4$ the color 3. It is easy to verify that this coloring makes $G'$ strongly conflict-free connected. Since the graph $G'$ is bipartite, the $k$-strong conflict-free connectivity problem is NP-complete even for the bipartite case.                         $\Box$

\end{document}